\documentclass{amsart}

\usepackage{amsmath}
\usepackage{amssymb}
\usepackage{amsthm}
\usepackage{mathrsfs}
\usepackage[all]{xy} 
\usepackage{bm}
\usepackage{comment}
\usepackage{color}
\usepackage{manfnt}
\usepackage{indentfirst}

\newtheorem{theorem}{Theorem}[section]
\newtheorem{lemma}[theorem]{Lemma}
\newtheorem{prop}[theorem]{Proposition}

\theoremstyle{definition}
\newtheorem{definition}[theorem]{Definition}
\newtheorem{remarks}[theorem]{Remark}

\numberwithin{equation}{section}

\allowdisplaybreaks[4]

\title[Alternating multiple poly-Bernoulli numbers]{Alternating variants of multiple poly-Bernoulli numbers and finite multiple zeta values in characteristic $0$ and $p$}
\author{Daichi Matsuzuki}
\email{m19044h@math.nagoya-u.ac.jp}
\address{Graduate School of Mathematics, Nagoya University, 
Furo-cho, Chikusa-ku, Nagoya, 464-8602, Japan}
\date{April 16, 2021.}

\begin{document}
\maketitle
\begin{abstract}
This paper consists of two parts: the characteristic $0$ part and the characteristic $p$ part.
In characteristic $0$ part, we introduce an alternating extension of multiple poly-Bernoulli numbers of K.~Imatomi, M.~Kaneko and E.~Takeda and obtain explicit presentations of the alternating finite multiple zeta values introduced by J. Zhao in term of the alternating extension of multiple poly-Bernoulli numbers.
In characteristic $p$ part, we introduce positive characteristic analogues of alternating finite multiple zeta values and express them as special values of finite Carlitz multiple polylogarithms defined by C.-Y.~Chang and Y.~Mishiba. We introduce alternating variants of R. Harada's multiple poly-Bernoulli-Carlitz numbers, which are analogues of multiple poly-Bernoulli numbers, to obtain explicit presentations of the finite alternating multiple zeta values.
We show that any finite multiple zeta value with integer index is expressed as $k$-linear combination of FMZV's with all-positive indices.
\end{abstract}
\tableofcontents
\section{Introduction}

In this paper, we generalize results of K.~Imatomi, M.~Kaneko and E.~Takeda on multiple poly-Bernoulli numbers in characteristic $0$ and their characteristic $p$ analogues ($p$: prime) established by R.~Harada to alternating setting.

Multiple poly-Bernoulli numbers (cf. \cite{Imatomi2014}) are generalization of poly-Bernoulli numbers, which are generalization of Bernoulli numbers (\cite{Kaneko1997}). They obtained connections of multiple poly-Bernoulli numbers with Stirling numbers and finite multiple zeta values.
We generalize these results of to alternating setting (Theorems \ref{alBernoulliandStirling} and \ref{0result}).

In characteristic $p$ case, L.~Carlitz introduced analogues of Bernoulli numbers called Bernoulli-Carlitz numbers. Harada generalized the notion to multiple poly-Bernoulli-Carlitz numbers and established their relationship with analogues of Stirling numbers and of finite multiple zeta values, respectively introduced in \cite{Kaneko2016} and \cite{Chang2017}.
We further generalize his results to alternating setting; alternating variants of multiple poly-Bernoulli-Carlitz numbers are introduced in Definition \ref{DefMPBCN} and we obtain explicit presentations of the alternating finite multiple zeta values in terms of them (Theorems \ref{thmMPBCNandStirling} and \ref{FAMZVandMPBCN}).

In appendix \ref{Finite multiple zeta values with non-all-positive indices}, we show that FMZV with integer index is expressed as $k$-linear combination of FMZV's with all-positive indices.

\section{Characteristic $0$}\label{Characteristic $0$}
In this section, we discuss the characteristic $0$ part. In \S \ref{Review on the result in original (non-alternating) case}, we review the results of Imatomi, Kaneko and Takeda: the connections of multiple poly-Bernoulli numbers with Stirling numbers and finite multiple zeta values. In \S \ref{0alternatingsection}, we consider alternating extension of their results. We introduce alternating multiple poly-Bernoulli numbers (Definition \ref{altMPBN}) and obtain their relationships with Stirling numbers and alternating finite multiple zeta values (Theorems \ref{alBernoulliandStirling} and \ref{0result}).

\subsection{Review on the results in original (non-alternating) case}\label{Review on the result in original (non-alternating) case}


For $\bold{s} \in \mathbb{Z}^r$ ($r\in \mathbb{N}$ is arrowed to be $0$), \textit{the multiple polylogarithm} (\textit{MPL} for short) $\mathrm{Li}_{\bold{s}}(z_1,\,\dots,\,z_r)$ is the multivariable series defined by

\begin{align}
\mathrm{Li}_{\bold{s}}(z_1,\,\dots,\,z_r)=\sum_{n_1>\cdots>n_r\geq1}\frac{z_1^{n_1}\cdots z_r^{n_r}}{n_1^{s_1}\cdots n_r^{s_r}},
\end{align}
(cf. \cite[Definition 2.3.1]{ZhaoBook}). We define $\mathrm{Li}_{\bold{s}}(z):=1$ if $r=0$ by convention.

For  $\bold{s} \in \mathbb{Z}^r$ and $n \in \mathbb{N}$, Imatomi, Kaneko and Takeda (\cite{Imatomi2014}) defined rationals $B_n^\bold{s}$ and $C_n^\bold{s}$ called the \textit{multiple poly-Bernoulli numbers} (\textit{MPBNs} for short) by
\begin{align*}
\sum_{n\geq0}B_n^{\bold{s}}\frac{x^n}{n!}&=\frac{\mathrm{Li}_{\bold{s}}(1-e^{-x},\,1,\,\dots,\,1)}{1-e^{-x}},\quad 
\sum_{n\geq0}C_n^\bold{s}\frac{x^n}{n!}&=e^{-x}\frac{\mathrm{Li}_{\bold{s}}(1-e^{-x},\,1,\,\dots,\,1)}{1-e^{-x}}.
\end{align*}
If $r=1$, then $B_n^\bold{s}$ and $C_n^\bold{s}$ coincide with poly-Bernoulli numbers defined in(\cite{Kaneko1997, Arakawa1999}; if $s_1=1$ in addition, these are equal to Bernoulli numbers.

\vspace{0.1in}
Imatomi, Kaneko and Takeda proved the following equations:

\begin{prop}[{\cite[Proposition 5]{Imatomi2014}}]\label{0dual}

We have the equalities 
\begin{align*}
&\sum_{s_1,\,\dots,\,s_r\geq 1}\sum_{n\geq0}B_n^{(-s_1,\,-s_2,\,\dots,\,-s_r)} \frac{x^n}{n!}\frac{y_1^{s_1}}{s_1!}\cdots \frac{y_r^{s_r}}{s_r!} \\
=&\frac{(1-e^{-x})^{r-1}}{( e^{-y_1}+e^{-x}-1)(e^{-y_1-y_2}+e^{-x}-1)\cdots( e^{-y_1-\cdots-y_r}+e^{-x}-1)} 
\end{align*}
and
\begin{align*}
&\sum_{s_1,\,\dots,\,s_r\geq 1}\sum_{n\geq0} C_n^{(-s_1,\,-s_2,\,\dots,\,-s_r)} \frac{x^n}{n!}\frac{y_1^{s_1}}{s_1!}\cdots \frac{y_r^{s_r}}{s_r!} \\
=&\frac{e^{-x}(1-e^{-x})^{r-1}}{(e^{-y_1}+e^{-x}-1)(e^{-y_1-y_2}+e^{-x}-1)\cdots( e^{-y_1-\cdots-y_r}+e^{-x}-1)} .
\end{align*}

\end{prop}
\subsubsection{Connection with the Stirling numbers}
For $m,\,n \in \mathbb{N}$, the \textit{Stirling numbers} 
$\left[
    \begin{array}{c}
     n\\
     m
    \end{array}
  \right]
$
,
$\left\{
    \begin{array}{c}
     n\\
     m
    \end{array}
  \right\}\in \mathbb{Z}
$
of the \textit{first} and the \textit{second kind} are defined by the formulae

\begin{align*}
x(x+1)\cdots(x+n-1)=\sum_{m\geq0}^n \left[
    \begin{array}{c}
     n\\
     m
    \end{array}
  \right]x^m, \label{defStirling0}\\
(e^x-1)^m=m!\sum_{n\geq m}\left\{
\begin{array}{c}
     n\\
     m
    \end{array}
  \right\}
  \frac{x^n}{n!}.
  \end{align*}
(see \cite[\S 6.1 and \S 7.4 (7.49)]{GrahamBook}). In this paper, we use the formula
\begin{equation}
e^x(e^x-1)^{m-1}=(m-1)!\sum_{n\geq m-1}\left\{
\begin{array}{c}
     n+1\\
     m
    \end{array}
  \right\}
  \frac{x^n}{n!}, \label{diffdef}
  \end{equation}
which is obtained by the differentiation, and the duality 
\begin{equation}
\left[
    \begin{array}{c}
     n\\
     m
    \end{array}
  \right]
  \equiv
  \left\{
\begin{array}{c}
     l-m\\
     l-n
    \end{array}
  \right\} \mod l
   \label{Stirlingduality}
\end{equation}
between two kinds of Stirling numbers, which holds for prime $l$ and $1\leq m<n<l$ (see \cite[\S 5]{Hoffman2015}).
Using these integers, we can  write multiple poly-Bernoulli numbers down as finite sums:

\begin{theorem}[{\cite[Theorem 3]{Imatomi2014}}] \label{Imatomithem}
For $\bold{s} =(s_1,\cdots,s_r) \in \mathbb{Z}^r$, we have for $n \in \mathbb{N}$ the equalities
\begin{align}
B_n^{\bold{s}}=(-1)^n \sum_{n+1\geq m_1>\cdots>m_r>0} \frac{(-1)^{m_1-1}(m_1-1)!\left\{
    \begin{array}{c}
     n\\
     m_1-1
    \end{array}
  \right\}}
{m_1^{s_1}\cdots m_r^{s_r}},\\
C_n^{\bold{s}}=(-1)^n \sum_{n+1\geq m_1>\cdots>m_r>0} \frac{(-1)^{m_1-1}(m_1-1)!\left\{
    \begin{array}{c}
     n+1\\
     m_1
    \end{array}
  \right\}}
{m_1^{s_1}\cdots m_r^{s_r}}.
\end{align}
\end{theorem}

\subsubsection{Connection with finite multiple zeta values}

The ring $\prod_l(\mathbb{Z}/l\mathbb{Z})/\bigoplus_l (\mathbb{Z}/l\mathbb{Z})$, where the symbol $l$ runs through the set of all prime numbers is denoted by $\mathcal{A}$. 
It should be noticed that the field $\mathbb{Q}$ can be canonically embedded into the ring $\mathcal{A}$, that is, the ring $\mathcal{A}$ is a $\mathbb{Q}$-algebra.

\begin{definition}[{\cite[\S 7]{Kaneko2019}}]\label{0FMZVs}
For $\bold{s}\in \mathbb{Z}^r$, the element $\zeta_{\mathcal{A}}(\bold{s})$$=($$\zeta_{\mathcal{A}}(\bold{s})_l)_{\text{$l$:prime}}$ of $\mathcal{A}$ called the \textit{finite multiple zeta value} (\textit{FMZV} for short) is defined to be the image under the surjection $\prod_l(\mathbb{Z}/l\mathbb{Z}) \rightarrow \mathcal{A}$ of the elements of $\prod_l(\mathbb{Z}/l\mathbb{Z})$ whose component in the direct factor $\mathbb{Z}/l\mathbb{Z}$ is
\begin{equation*}
\zeta_{\mathcal{A}}(\bold{s})_l :=\sum_{l>n_1>\cdots>n_r>0}\frac{1}{n_1^{s_1}\,\cdots n_r^{s_r}} \in \mathbb{Z}/l\mathbb{Z}.
\end{equation*}
\end{definition}

We call the natural number $r$ the \textit{depth} of the FMZV $\zeta_{\mathcal{A}}(\bold{s})$. The product of any two FMZVs is written by a $\mathbb{Q}$-linear combination of FMZVs (\cite[Section 7]{Kaneko2019}). FMZVs are realized as special values of the finite version of multiple polylogarithm $\mathfrak{L}_{\mathcal{A},\,\bold{s}}(z)$ introduced by K.~Sakugawa, S.-I.~Seki (\cite[Definition 3.8]{Sakugawa2017}), whose value $\mathfrak{L}_{\mathcal{A},\,\bold{s}}(\bold{a}_1,\,\dots,\,\bold{a}_r)$=$(\mathfrak{L}_{\mathcal{A},\,\bold{s}}(\bold{a}_1,\,\dots,\,\bold{a}_r))_l$ (each $\bold{a}_i=(a_{i,\,l})_{l:\text{prime}}$ is an element of $\mathcal{A}$) is given by
\begin{equation}
(\mathfrak{L}_{\mathcal{A},\,\bold{s}}(\bold{a}_1,\,\dots,\,\bold{a}_r))_l=\sum_{l>n_1>\cdots>n_r>0}\frac{a_{1,\,l}^{n_1}\cdots a_{r,\,l}^{n_r}}{n_1^{s_1}\,\cdots n_r^{s_r}} \in \mathbb{Z}/l\mathbb{Z} \label{FMPL}
\end{equation}
for each prime $l$, in precise we have the equality

\begin{equation}
\zeta_\mathcal{A}(\bold{s})=\mathfrak{L}_{\mathcal{A},\,\bold{s}}(1,\,\dots,\,1) \label{Seki}
\end{equation}
for each $\bold{s} \in \mathbb{Z}^r$.

Imatomi, Kaneko and Takeda obtained the following equalities:

\begin{theorem}[{\cite[Theorem 8]{Imatomi2014}}]\label{thmImatomi}
\begin{enumerate}
\item
For a prime $l$ and $\bold{s} \in \mathbb{Z}^r$, we have the congruence
\begin{equation*}
\zeta_{\mathcal{A}}(\bold{s})_l\equiv -C_{l-2}^{s_1 -1,\,s_2,\,\dots,\,s_r} \mod l. \label{eqImatomi1} 
\end{equation*}
\item
If we take $r^\prime \in \mathbb{N}$ and put $\bar{\bold{s}}:=(1,\,\dots,1,$ $s_1,\,\dots,\,s_r )\in \mathbb{N}_{>0}^{r+r^\prime}$, then we have the congruence 
\begin{equation}
\zeta_{\mathcal{A}}(\bar{\bold{s}})_l\equiv -C_{l-r^\prime-2}^{s_1 -1,\,s_2,\,\dots,\,s_r} \mod l. \label{eqImatomi2} 
\end{equation}
\end{enumerate}
\end{theorem}

\subsection{Alternating version}\label{0alternatingsection}


In this subsection, we consider alternating extensions of notions and results in the previous subsection. We first give an alternating extension to the MPBNs by the following series:

\begin{definition}\label{altMPBN}
For $\bold{s} \in \mathbb{Z}^r$ and $\boldsymbol{\epsilon} \in (\mathbb{Z}^\times)^r=\{ \pm1\}^r$, the sequences of rationals $B_n^\bold{s;\,\boldsymbol{\epsilon}}$ and $C_n^\bold{s;\,\boldsymbol{\epsilon}}$  are defined by the following series:

\begin{align}
\sum_{n\geq0}B_n^{\bold{s};\,\boldsymbol{\epsilon}}\frac{x^n}{n!}&=\frac{\mathrm{Li}_{\bold{s}}((1-e^{-x})\epsilon_1,\,\epsilon_2,\,\dots,\,\epsilon_r)}{1-e^{-x}},\label{DefAMPBernoulliNs}\\
\sum_{n\geq0}C_n^{\bold{s};\,\boldsymbol{\epsilon}}\frac{x^n}{n!}&=e^{-x}\frac{\mathrm{Li}_{\bold{s}}((1-e^{-x})\epsilon_1,\,\epsilon_2,\,\dots,\,\epsilon_r)}{1-e^{-x}}.\label{DefAMPCernoulliNs}
\end{align}
We call these rationals \textit{alternating multiple poly-Bernoulli numbers} (\textit{AMPBNs} for short).
\end{definition}

The following could be said as alternating extension of Proposition \ref{0dual}:
\begin{prop}
For $\boldsymbol{\epsilon} \in \{ \pm1\}^r$, the following equalities hold:
\begin{align*}
&\sum_{s_1,\,\dots,\,s_r\geq 0}\sum_{n\geq0}B_n^{(-s_1,\,-s_2,\,\dots,\,-s_r;\,\boldsymbol{\epsilon})} \frac{x^n}{n!}\frac{y_1^{s_1}}{s_1!}\cdots \frac{y_r^{s_r}}{s_r!} \\
=&\frac{(1-e^{-x})^{r-1}}{(\epsilon_1 e^{-y_1}+e^{-x}-1)(\epsilon_1 \epsilon_2 e^{-y_1-y_2}+e^{-x}-1)\cdots(\epsilon_1\,\cdots\epsilon_r e^{-y_1-\cdots-y_r}+e^{-x}-1)} \\
\intertext{and}
&\sum_{s_1,\,\dots,\,s_r\geq 0}\sum_{n\geq0} C_n^{(-s_1,\,-s_2,\,\dots,\,-s_r;\,\boldsymbol{\epsilon})} \frac{x^n}{n!}\frac{y_1^{s_1}}{s_1!}\cdots \frac{y_r^{s_r}}{s_r!} \\
=&\frac{e^{-x}(1-e^{-x})^{r-1}}{(\epsilon_1 e^{-y_1}+e^{-x}-1)(\epsilon_1 \epsilon_2 e^{-y_1-y_2}+e^{-x}-1)\cdots(\epsilon_1\,\cdots\epsilon_r e^{-y_1-\cdots-y_r}+e^{-x}-1)} .
\end{align*}
\end{prop}

\begin{proof}
By the equality \eqref{DefAMPBernoulliNs}, we have
\begin{align*}
&\sum_{s_1,\,\dots,\,s_r\geq 0}\sum_{n\geq0}B_n^{(-s_1,\,-s_2,\,\dots,\,-s_r;\,\boldsymbol{\epsilon})} \frac{x^n}{n!}\frac{y_1^{s_1}}{s_1!}\cdots \frac{y_r^{s_r}}{s_r!} \\
=&\sum_{s_1,\,\dots,\,s_r\geq 0}\frac{\mathrm{Li}_{(-s_1,\,-s_2,\,\dots,\,-s_r)}((1-e^{-x})\epsilon_1,\,\epsilon_2,\,\dots,\,\epsilon_r)}{1-e^{-x}}\frac{y_1^{s_1}}{s_1!}\cdots \frac{y_r^{s_r}}{s_r!} \\
=&\sum_{s_1,\,\dots,\,s_r\geq 0}(1-e^{-x})^{-1}\sum_{m_1>\cdots>m_r>0}\frac{\epsilon_1^{m_1}\cdots{\epsilon}_r^{m_r}(1-e^{-x})^{m_1}}{m_1^{-s_1}\cdots m_r^{-s_r}}    \frac{y_1^{s_1}}{s_1!}\cdots \frac{y_r^{s_r}}{s_r!} \\
=&(1-e^{-x})^{-1} \sum_{m_1>\cdots>m_r>0}\epsilon_1^{m_1}\cdots{\epsilon}_r^{m_r}(1-e^{-x})^{m_1} \sum_{s_1,\,\dots,\,s_r\geq 0}\frac{(m_1 y_1)^{s_1}}{s_1!}\cdots\frac{(m_r y_r)^{s_r}}{s_r!}\\
=&(1-e^{-x})^{-1} \sum_{m_1>\cdots>m_r>0}\epsilon_1^{m_1}\cdots{\epsilon}_r^{m_r}(1-e^{-x})^{m_1} e^{m_1 y_1}\cdots e^{m_r y_r};
\end{align*}
in what follows we continue calculation by putting $n_1:=m_1-m_2,\, n_2:=m_2-m_3,\,\dots,\, n_{r-1}:=m_{r-1}-m_r$ and $n_r:=m_r$
\begin{align*}
=&(1-e^{-x})^{-1} \sum_{n_1,\,\dots,\,n_r\geq1}\epsilon_1^{n_1+\cdots+n_r}\cdots{\epsilon}_r^{n_r}(1-e^{-x})^{n_1+\cdots+n_r} e^{(n_1+\cdots+n_r) y_1}\cdots e^{n_r y_r}\\
=&(1-e^{-x})^{-1} \sum_{n_1,\,\dots,\,n_r\geq1} \{\epsilon_1 (1-e^{-x})e^{y_1}\}^{n_1}   \{\epsilon_1\epsilon_2 (1-e^{-x})e^{y_1+y_2}\}^{n_2}\cdots\\
&\cdots \{\epsilon_1\cdots \epsilon_r (1-e^{-x})e^{y_1+\cdots+y_r}\}^{n_r} \\
=&(1-e^{-x})^{-1} \frac{\epsilon_1 (1-e^{-x})e^{y_1}}{1-(\epsilon_1 (1-e^{-x})e^{y_1})} \frac{\epsilon_1\epsilon_2 (1-e^{-x})e^{y_1+y_2}}{1-(\epsilon_1\epsilon_2 (1-e^{-x})e^{y_1+y_2})}\cdots\\
&\cdots\frac{\epsilon_1\cdots \epsilon_r (1-e^{-x})e^{y_1+\cdots+y_r}}{1-(\epsilon_1\cdots \epsilon_r (1-e^{-x})e^{y_1+\cdots+y_r})}\\
=&(1-e^{-x})^{-1} \frac{ 1-e^{-x}}{\epsilon_1 e^{-y_1}+e^{-x}-1} \cdots
\frac{ 1-e^{-x}}{\epsilon_1\,\cdots\epsilon_r e^{-y_1-\cdots-y_r}+e^{-x}-1}\\
=&\frac{(1-e^{-x})^{r-1}}{(\epsilon_1 e^{-y_1}+e^{-x}-1)(\epsilon_1 \epsilon_2 e^{-y_1-y_2}+e^{-x}-1)\cdots(\epsilon_1\,\cdots\epsilon_r e^{-y_1-\cdots-y_r}+e^{-x}-1)}, 
\end{align*}
(note $\epsilon_i^{-1}=\epsilon_i$ because $\epsilon_i=\pm1$). Hence the first equality holds. By the equality \eqref{DefAMPCernoulliNs}, we have
\begin{align*}
&\sum_{s_1,\,\dots,\,s_r\geq 0}\sum_{n\geq0} C_n^{(-s_1,\,-s_2,\,\dots,\,-s_r;\,\boldsymbol{\epsilon})} \frac{x^n}{n!}\frac{y_1^{s_1}}{s_1!}\cdots \frac{y_r^{s_r}}{s_r!}\\
=&e^{-x}\sum_{s_1,\,\dots,\,s_r\geq 0}\sum_{n\geq0} B_n^{(-s_1,\,-s_2,\,\dots,\,-s_r;\,\boldsymbol{\epsilon})} \frac{x^n}{n!}\frac{y_1^{s_1}}{s_1!}\cdots \frac{y_r^{s_r}}{s_r!}\\
=&\frac{e^{-x}(1-e^{-x})^{r-1}}{(\epsilon_1 e^{-y_1}+e^{-x}-1)(\epsilon_1 \epsilon_2 e^{-y_1-y_2}+e^{-x}-1)\cdots(\epsilon_1\,\cdots\epsilon_r e^{-y_1-\cdots-y_r}+e^{-x}-1)}.
\end{align*}
\end{proof}
%

\subsubsection{Connection with Stirling numbers}

The following is an alternating extension of Theorem \ref{Imatomithem}:

\begin{theorem} \label{alBernoulliandStirling}
If we take $\bold{s} \in \mathbb{Z}^r$ and $\boldsymbol{\epsilon} \in \{ \pm1\}^r$ as in Definition \ref{altMPBN}; then we have the following equality for $n\in \mathbb{N}_{>0}$:
\begin{equation}
C_n^{\bold{s};\,\boldsymbol{\epsilon}}=(-1)^n \sum_{n+1\geq m_1>\cdots>m_r>0} \frac{\epsilon_1^{m_1}\cdots\epsilon_r^{m_r}(-1)^{m_1-1}(m_1-1)!\left\{
    \begin{array}{c}
     n+1\\
     m_1
    \end{array}
  \right\}}
{m_1^{s_1}\cdots m_r^{s_r}}.
\end{equation}
\end{theorem}

\begin{proof}
By the formula \eqref{diffdef}, we have

\begin{align*}
&\sum_{n\geq0}C_n^{\bold{s};\,\boldsymbol{\epsilon}}\frac{x^n}{n!}=e^{-x}\frac{\mathrm{Li}_{\bold{s}}((1-e^{-x})\epsilon_1,\,\epsilon_2,\,\dots,\,\epsilon_r)}{1-e^{-x}}\\
=&\sum_{m_1>\cdots>m_r>0}\frac{e^{-x}(1-e^{-x})^{m_1-1} \epsilon_1^{m_1}\cdots \epsilon_r^{m_r}}{m_1^{s_1}\cdots m_r^{s_r}}\\
=&\sum_{m_1>\cdots>m_r>0}\frac{(m_1-1) !\epsilon_1^{m_1}\cdots \epsilon_r^{m_r}(-1)^{m_1-1}}{m_1^{s_1}\cdots m_r^{s_r}} \frac{e^{-x}(e^{-x}-1)^{m_1-1}}{(m_1-1)!}\\
=&\sum_{m_1>\cdots>m_r>0}\frac{(m_1-1)! \epsilon_1^{m_1}\cdots \epsilon_r^{m_r}(-1)^{m_1-1}}{m_1^{s_1}\cdots m_r^{s_r}} \sum_{n\geq m_1-1}
\left\{
    \begin{array}{c}
     n+1\\
     m_1
    \end{array}
  \right\}
\frac{(-x)^n}{n!}\\
=&\sum_{n\geq0}x^n (-1)^n \sum_{n+1\geq m_1>\cdots>m_r>0} \frac{\epsilon_1^{m_1}\cdots\epsilon_r^{m_r}(-1)^{m_1-1}(m_1-1)!\left\{
    \begin{array}{c}
     n+1\\
     m_1
    \end{array}
  \right\}}
{m_1^{s_1}\cdots m_r^{s_r}}.
\end{align*}
Then, comparing the coefficients $x^n$ for each $n$ results in the desired equality.
\end{proof}

\subsubsection{Connection with finite multiple zeta values}

For each $\bold{s}\in \mathbb{Z}^r$ and $\boldsymbol{\epsilon} \in \{\pm1\}^r$, the element $\zeta_{\mathcal{A}}(\bold{s};\,\boldsymbol{\epsilon})=(\zeta_{\mathcal{A}}(\bold{s};\,\boldsymbol{\epsilon})_l)_{\text{$l$:prime}}$ of $\mathcal{A}$ is defined by
\begin{equation*}
\zeta_{\mathcal{A}}(\bold{s};\,\boldsymbol{\epsilon})_l :=\sum_{l>n_1>\cdots>n_r>0}\frac{\epsilon_1^{n_1}\,\cdots\,\epsilon_r^{n_r}}{n_1^{s_1}\,\cdots n_r^{s_r}} \in \mathbb{Z}/l\mathbb{Z}.
\end{equation*}
We call these elements of $\mathcal{A}$ \textit{alternating finite multiple zeta values} (\textit{AFMZVs} for short). When $\bold{s} \in \mathbb{N}_{>0}^r$, these are special examples with the superbility $1$ of \textit{finite Euler sums} introduced by J.~Zhao (\cite{Zhao2015}). Sakugawa and Seki also treat these elements in \cite{Sakugawa2017}.

We have the following equality 

\begin{equation}
\zeta_\mathcal{A}(\bold{s};\,\boldsymbol{\epsilon})=\mathfrak{L}_{\mathcal{A},\,\bold{s}}(\epsilon_1,\,\dots,\,\epsilon_r) \label{alSeki}
\end{equation}
for each $\bold{s}$ and $\boldsymbol{\epsilon}$. This is an alternating extension of the equality \eqref{Seki}.

The following is a alternating extension of Theorem \ref{thmImatomi}.

\begin{theorem} \label{0result}
\begin{enumerate}
\item \label{(1)1.8}
For $r\in \mathbb{N}_{>0}$, $\bold{s}=(s_1,\,\dots,\,s_r)$ and $\boldsymbol{\epsilon}$=$(\epsilon_1$, $\dots$, $\epsilon_r)$, we have the following congruence for each odd prime $l$:
\begin{equation*}
\zeta_{\mathcal{A}}(\bold{s};\, \boldsymbol{\epsilon})_l\equiv -C_{l-2}^{(s_1 -1,\,s_2,\,\dots,\,s_r);\,\boldsymbol{\epsilon}} \text{  mod $l$}. \label{eqImatomi1} 
\end{equation*}

\item \label{(2)1.8}
For ${r^\prime} \in \mathbb{N}$, $\bar{\bold{s}} =(1,\,\dots,1,\,s_1,\,\dots,\,s_r)= \mathbb{Z}^{r+r^\prime}$ and $\bar{\boldsymbol{\epsilon}}$$=(1,\,\dots,\,1,\,\epsilon_1,\,\dots,$$\,\epsilon_r)$$ \in \{\pm1\}^{r+r^\prime}$, the following congruence holds for each odd prime $l$:
\begin{equation}
\zeta_{\mathcal{A}}(\bar{\bold{s}},\,\bar{\boldsymbol{\epsilon}})_l\equiv -C_{l-r^\prime-2}^{(s_1 -1,\,s_2,\,\dots,\,s_r);\,\boldsymbol{\epsilon}} \mod l. \label{eqImatomi2} 
\end{equation}
\end{enumerate}
\end{theorem}

\begin{proof}
First we note that the equality
\begin{equation*}
 m!\left\{
    \begin{array}{c}
     l-1\\
     m
    \end{array}
  \right\}
=\sum_{s=0}^{m}
\left(
    \begin{array}{c}
     m\\
     s
    \end{array}
  \right) 
  (-1)^{m-s}s^{l-1}
\end{equation*}
holds for each positive integer $m$ (\cite[\S6.1, (6.19)]{GrahamBook}). It implies the congruence
                                                                                                                                                                                                                                                                                                                                                                                                                                                                                                                                                                                                                                                                                                                                                                                                                                                                                                                                                                                                                                                                                                                                                                                                                                                                                                                                                  \begin{equation}
                                                                                                                                                                                                                                                                                                                                                                                                                                                                                                                                                                                                                                                                                                                                                                                                                                                                                                                                                                                                                                                                                                                                                                                                                                                                                                                                                                                     (-1)^m m!\left\{
    \begin{array}{c}
     l-1\\
     m
    \end{array}
  \right\}
\equiv
\sum_{s=1}^{m}
\left(
    \begin{array}{c}
     m\\
     s
    \end{array}
  \right)(-1)^s
= (1-1)^m-1=-1, \mod l  \label{ST}
   \end{equation}                                                                                                                                                                                                                                                                                                                                                                                                                                                                                                                                                                                                                                                                                                                                                                                                                                                                                                                                                                                                                                                                                                                                                                                                                                                                                                                                                                         since we have $s^{l-1}\equiv 1$ for $1\leq s\leq m<l$. By the congruence \eqref{ST}, we have
\begin{align*}
C_{l-2}^{(s_1-1,\,s_2,\,\dots,\,s_r);\,\boldsymbol{\epsilon}}
&=(-1)^{l-2} \sum_{l-1\geq m_1>\cdots>m_r>0} \frac{\epsilon_1^{m_1}\cdots\epsilon_r^{m_r}(-1)^{m_1-1}(m_1-1)!\left\{
    \begin{array}{c}
     l-1\\
     m_1
    \end{array}
  \right\}}
{m_1^{s_1-1}\cdots m_r^{s_r}}\\
&= \sum_{l-1\geq m_1>\cdots>m_r>0} \frac{\epsilon_1^{m_1}\cdots\epsilon_r^{m_r}(-1)^{m_1}(m_1)!\left\{
    \begin{array}{c}
     l-1\\
     m_1
    \end{array}
  \right\}}
{m_1^{s_1}\cdots m_r^{s_r}}\\
&\equiv \sum_{l-1\geq m_1>\cdots>m_r>0} \frac{\epsilon_1^{m_1}\cdots\epsilon_r^{m_r}(-1)}
{m_1^{s_1}\cdots m_r^{s_r}}
=-\zeta_{\mathcal{A}}(\bold{s};\, \boldsymbol{\epsilon})_l \mod l,
\end{align*}
hence we obtain the assertion \eqref{(1)1.8}.

We have
\begin{align*}
\zeta_{\mathcal{A}}(\bar{\bold{s}},\,\bar{\boldsymbol{\epsilon}})_l
&=\sum_{l>i_1>\cdots>i_{r^\prime}>m_1>\cdots>m_r>0}\frac{\epsilon_1^{m_1}\,\cdots\,\epsilon_r^{m_r}}{i_1 \cdots i_{r^\prime} m_1^{s_1}\,\cdots m_r^{s_r}}\\
&=\sum_{l-r^\prime>m_1>\cdots>m_r>0}\frac{\epsilon_1^{m_1}\,\cdots\,\epsilon_r^{m_r}}{m_1^{s_1}\,\cdots m_r^{s_r}} \sum_{l>i_1 >\cdots >i_{r^{\prime}} >m_1} \frac{1}{i_1 \cdots i_{r^\prime}}\\
&\equiv \sum_{l-r^\prime>m_1>\cdots>m_r>0}\frac{\epsilon_1^{m_1}\,\cdots\,\epsilon_r^{m_r}}{m_1^{s_1}\,\cdots m_r^{s_r}} \sum_{l-m_1>i_{r^\prime} >\cdots >i_1 \geq1} \frac{(-1)^{r^\prime}}{i_1 \cdots i_{r^\prime}}.\\
\end{align*}
The congruence of generating series
\begin{align*}
&\sum_{m=0}^{l-m_1-1} \left\{ \sum_{l-m_1-1\geq i_m>\cdots>i_1\geq1}\frac{(-1)^{r^\prime}}{i_1\cdots i_m} \right\}x^{m+1}\\
\equiv&(-1)^{r^\prime}\sum_{m=0}^{l-m_1-1} \left\{ \frac{1}{(l-m_1-1)!} \sum_{l-m_1-1\geq j_1>\cdots>j_{N-m} \geq1}{j_1\cdots j_{N-m}} \right\}x^{m+1}\mod l\\
=&\frac{(-1)^{r^\prime}}{(l-m_1-1)!}x(x+1)\cdots(x+l-m_1-1)
=\frac{(-1)^{r^\prime}}{(l-m_1-1)!}\sum_{m\geq0}^{l-m_1-1}\left[
    \begin{array}{c}
     l-m_1\\
     m+1
    \end{array}
  \right]x^{m+1}\\
\equiv&\frac{(-1)^{r^\prime}}{(l-m_1-1)!}\sum_{m\geq0}^{l-m_1-1}\left\{
    \begin{array}{c}
     l-m-1\\
    m_1
   \end{array}
  \right\}x^{m+1} \mod l\\
\equiv&(-1)^{r^\prime+m_1+1}m_1!
\sum_{m\geq0}^{l-m_1-1}\left\{
    \begin{array}{c}
     l-m-1\\
    m_1
   \end{array}
  \right\}x^{m+1} \mod l  
\end{align*}
follows from \eqref{Stirlingduality}. Hence we obtain
\begin{align*}
\zeta_{\mathcal{A}}(\bar{\bold{s}},\,\bar{\boldsymbol{\epsilon}})_l
&\equiv
\sum_{l-r^\prime>m_1>\cdots>m_r>0}\frac{\epsilon_1^{m_1}\,\cdots\,\epsilon_r^{m_r} (-1)^{m_1+r^\prime+1} m_1 ! 
\left\{
    \begin{array}{c}
     l-r^\prime-1\\
     m_1
    \end{array}
  \right\}}{m_1^{s_1}\,\cdots m_r^{s_r}}\\
&=-C_{l-r^\prime-2}^{(s_1 -1,\,s_2,\,\dots,\,s_r);\,\boldsymbol{\epsilon}}.
\end{align*}
\end{proof}

\section{Characteristic $p$}

In this section, we consider the characteristic $p$ analogues of the notions and results in the previous section. After the review on results in \cite{Harada2018} of Harada, we generalize his results to alternating setting. We introduce alternating extension of multiple poly-Bernoulli-Carlitz numbers (Definition \ref{DefMPBCN}) and establish their connection with Stirling-Carlitz numbers (Theorems \ref{thmMPBCNandStirling}). In Theorems \ref{FAMZVandMCPL}, we write alternating extension of finite multiple zeta values down in terms of special values of finite Carlitz multiple polylogarithm defined in \cite{Chang2017}. We obtain the relationship between alternating extensions of multiple poly-Bernoulli-Carlitz numbers and finite multiple zeta values (Theorem \ref{FAMZVandMPBCN}).

\subsection{Review on Harada's multiple poly-Bernoulli numbers} \label{Review on Harada's MPBCNs}


We fix a prime $p$ and its power $q$. The symbol $A$ denotes the polynomial ring $\mathbb{F}_q[\theta]$ in $\theta$ over the finite field $\mathbb{F}_q$ of $q$ elements and $k$ stands for the field $\mathbb{F}_q(\theta)$ of rational functions.

For each $n \in\mathbb{N}_{>0}$, the element $\theta^{q^n}-\theta$ of the set $A_+$ (of all monic polynomials) is denoted by $[n]$. We put
$D_n:=[n]^{q^0}[n-1]^{q^1}\cdots[1]^{q^{n-1}}\in A_+$, 
$L_n:=[n][n-1]\cdots[1](-1)^n \in A$ for $n\geq1$ and $D_0= L_0:=1$ (\cite{Carlitz1935,GossBook}).

For $n \in \mathbb{N}$ with the $q$-adic expansion
$
n=\sum_{j=0}^{d} \alpha_j q^j \quad(0\leq \alpha_j < q)
$,
we put
$
\Gamma_{n+1}:=\Pi (n):=\prod_{j=0}^{d} D_j^{\alpha_j} \in A_+$, which are called \textit{the Carlitz gamma} and \textit{the Carlitz factorial}  respectively, see \cite{Carlitz1935, GossBook}.
For each $d\in \mathbb{N}$ and $s\in\mathbb{Z}$, the sum 
$
\sum_{a} \frac{1}{a^s} \in k
$
(where $a$ runs through all monic polynomials of degree $d$ in $A$) is denoted by $S_d(s)$


Following \cite{Anderson1990}, we define polynomials $\mathfrak{H}_n(t,y)\in \mathbb{F}_q(t,y)\;(n\geq0)$ by 

\begin{equation*}
\sum_{n\geq0}\frac{\mathfrak{H}_n(t,y)}{\Gamma_{n+1}|_{\theta=t}}x^n=\left(1-\sum_{i\geq0}\frac{G_i(t,y)}{D_i|_{\theta=t}}x^{q^i}\right)^{-1} \in \mathbb{F}_q(t,\,y)[[x]] \label{eqDefofAT}
\end{equation*}
where $G_n(t,y):=\prod_{i=1}^{n}(t^{q^n}-y^{q^i})$ and we put $H_n(t) :=\mathfrak{H}_n(t,\theta) \in A[t]$; these are called the \textit{Anderson-Thakur polynomials}.
If we write
\begin{equation*}
H_{n}(t)=\sum_{j=0}^{m_{n+1}}u_{{n+1},\,j} t^{j} \text{,  with $u_{i,\,j} \in$ A and $m_{n+1}\neq 0$},
\end{equation*}
for $n\geq0$, then we have
\begin{equation*}
H_{n-1}^{{(d)}}(t)|_{t=\theta}:=\left(\sum_{j=0}^{m_n}u_{n,\,j}^{{q^d}} t^{j}\middle)\right|_{t=\theta}=L_d^{n}\Gamma_{n}S_d(n).
\end{equation*}
for any $d\in\mathbb{N}$ and $n \in \mathbb{N}_{>0}$, see \cite{Anderson1990}.
For each $\bold{s} =(s_1,\,\dots,\,s_r)\in \mathbb{N}_{>0}^r$, we put
\begin{equation*}
\mathfrak{J}_\bold{s}:=\{ \bold{j}=(j_1,\,\dots,\,j_r) \in \mathbb{N}^r \mid\text{$0\leq j_i \leq \deg_t H_{s_i-1}$ holds for any $1\leq i\leq r$.}\}
\end{equation*}
and denote $\theta^{j_1+\cdots+j_r}$ by $\theta^{\bold{j}}$ for short.

We define formal power series $e_C(z)$ (\cite{Carlitz1935}) and $\mathrm{Li}_\bold{s}(z_1,\,\dots,\,z_r)$ for all $\bold{s}=(s_1,\dots,s_r) \in \mathbb{N}^r$ (\cite{Chang2014}) respectively called the \textit{Carlitz exponential} and the \textit{Carlitz multiple polylogarithm} by
\begin{equation*}
e_C(z):=\sum_{i\geq 0}\frac{z^{q^i}}{D_i} \in k((z)), \  \mathrm{Li}_{\bold{s}}(z_1,\,\dots,\,z_r):=\sum_{i_1>\cdots>i_r\geq0}  \frac{z_1^{q^{i_1}}\cdots z_r^{q^{i_r}}}{L_{i_1}^{s_1}\cdots L_{i_r}^{s_r}}\in k((z_1,\dots,z_r)).
\end{equation*}
These are analogues of exponential and multiple polylogarithm functions.
\begin{definition}[{\cite[Definition 21]{Harada2018}}]\label{DefMPBCNoriginal}
For each $\bold{s}=(s_1,\dots,s_r) \in \mathbb{N}^r$ and $\bold{j}=(j_1,\,\dots,\,j_r) \in\mathfrak{J}_{\bold{s}}$, \textit{multiple poly-Bernoulli-Carlitz numbers} (\textit{MPBCNs} for short) $BC_n^{\bold{s},\,\bold{j}}$ are elements of $k$ defined by
\begin{equation*}
\sum_{n\geq0}BC_n^{\bold{s},\,\bold{j}}\frac{z^n}{\Pi(n)}:=\frac{\mathrm{Li}_{\bold{s}}(e_C(z)u_{s_1,\,j_1},\, u_{s_2,\,j_2},\dots,\, u_{s_r,\,j_r})}{e_C(z)}.
\end{equation*}
\end{definition}

The validity of the analogue of Proposition \ref{0dual} seems unclear since Harada's multiple poly-Bernoulli-Carlitz numbers are defined only in the case when $s_i$ are positive integers.

\subsubsection{Connection with Stirling-Carlitz numbers}

Let us recall the definition and some properties of the analogues of Stirling numbers (of the second kind) introduced in \cite{Kaneko2016}.

The \textit{Stirling-Carlitz numbers (of the second kind)} 
$
\left\{
    \begin{array}{c}
     n\\
     m
    \end{array}
  \right\}_C
$ ($n,\,m\in\mathbb{Z}$)
are defined by
\begin{equation*}
	\frac{(e_C(z))^m}{\Pi(m)}=\sum_{n\geq0}
	\left\{
	    \begin{array}{c}
	     n\\
	     m
	    \end{array}
	  \right\}_C
	\frac{z^n}{\Pi(n)}.
\end{equation*}
The definition is due to H.~Kaneko and T.~Komatsu (\cite{Kaneko2016}). We note that they also introduced analogues of the first kind Stirling numbers in \cite[\S 2]{Kaneko2016}.

We remind the following formulae required later:
\begin{align}
\left\{
    \begin{array}{c}
     n\\
     m
    \end{array}
  \right\}_C
  &=0 \text{  if $n<m$},  \quad 
\left\{
    \begin{array}{c}
     q^n-1\\
     q^m-1
    \end{array}
  \right\}_C
  =\begin{cases}
    0& n\neq m \\
    1& n=m
  \end{cases}. \label{Stirling2}
\end{align}
See \cite[(17)]{Kaneko2016} and \cite[(10)]{Harada2018} respectively for their proofs.

The following is an analogue of Theorem \ref{Imatomithem} obtained by Harada (\cite{Harada2018}), which describes MPBCNs as finite sums in terms of Stirling-Carlitz numbers.

\begin{theorem}[{\cite[Theorem 2.7]{Harada2018}}]\label{thmMPBCNandStirlingoriginal}
If $\bold{s}$ and $\bold{j}$ are as in the Definition \ref{DefMPBCNoriginal}, then the following equality in $k$ holds:
\begin{equation*}
BC_n^{\bold{s},\,\bold{j}}=\sum_{\log_q(n+1)\geq d_1>\cdots>d_r\geq0}\Gamma_{q^{d_1}}
\left\{
    \begin{array}{c}
     n\\
     q^{d_1}-1
    \end{array}
  \right\}_C \frac{u_{s_1,j_1}^{q^{d_1}}\cdots u_{s_r,j_r}^{q^{d_r}}}{L_{d_1}^{s_1}\cdots L_{d_r}^{s_r}}.
\end{equation*}

\end{theorem}

\subsubsection{Connection with finite multiple zeta values}
Characteristic $p$ analogues of FMZVs are introduced in \cite{Chang2017}. 
Let $\mathcal{A}_k$ be the quotient ring $\prod_P(A/(P))/\bigoplus_P (A/(P))$. Here, the symbol $P$ runs through the set $\operatorname{Spm}A$ of all monic irreducible polynomials in $A$. 
The ring $\mathcal{A}_k$ is naturally equipped with  $k$-algebra structure.

\begin{definition}[{\cite[${\S}$2]{Chang2017}}]\label{pFMZV}
For each $\bold{s} \in \mathbb{Z}^r$, the element $\zeta_{\mathcal{A}_k}(\bold{s})=(\zeta_{\mathcal{A}_k}(\bold{s})_P)_{P\in\operatorname{Spm}A}$ of $\mathcal{A}_k$ is defined by
\begin{equation*}
\zeta_{\mathcal{A}_k}(\bold{s})_P :\equiv\sum_{\substack{\deg P>\deg{a_1}>\cdots>\deg a_r\geq0\\a_i\text{ monic}}}\frac{1}{a_1^{s_1}\,\cdots a_r^{s_r}} \in A/(P);
\end{equation*}
these elements of $\mathcal{A}_k$ are called \textit{finite multiple zeta values} (\textit{FMZV} for short). We call the natural number $r$ the \textit{depth} of the FMZV $\zeta_{\mathcal{A}_k}(\bold{s})$.
\end{definition}

Chang and Mishiba introduced the \textit{finite Carlitz multiple polylogarithm}  $\mathrm{Li}_{\mathcal{A}_k,\,\bold{s}}(z)$ (\textit{FCMPL} for short) as a finite variant of CMPL.
For any $\textbf{s}=(s_1,s_2,\dots,s_r) \in \mathbb{N}^r$ and tuple $\bold{a}=((a_{P,\,1})_P,\,\dots,\,(a_{P,\,r})_P) \in \mathcal{A}_k^r$ with $a_{P,\,1}\in A/P$, the value $\mathrm{Li}_{\mathcal{A}_k,\,\bold{s}}(\bold{a})=(\mathrm{Li}_{\mathcal{A}_k,\,\bold{s}}(\bold{a})_P)_{P\in\operatorname{Spm}A}$ at $\bold{a}\in \mathcal{A}_k$ is given by
\begin{equation*}
\mathrm{Li}_{\mathcal{A}_k,\,\bold{s}}(\bold{a})_P:\equiv\sum_{\deg P>i_1>\cdots>i_r\geq0} \frac{a_{P,\,1}^{q^{i_1}}\,\cdots \,a_{P,\,r}^{q^{i_r}}}{L_{i_1}^{s_1}\,\cdots \, L_{i_r}^{s_r}} \in A/(P).
\end{equation*}
It is obvious that the value $\mathrm{Li}_{\mathcal{A}_k,\,\bold{s}}(\bold{a})$ is independent on the choices of representatives of $(a_{P,\,1}),\,\dots,\,(a_{P,\,r-1})$ and $(a_{P,\,r})$. See \cite[\S 3]{Chang2017} for the precise definition of the finite Carlitz multiple polylogarithm.

This definition of FCMPLs coincides with that in \cite{Chang2017} if restricted on the subset $k^r$ of $\mathcal{A}_k^r$.

Chang and Mishiba obtained the following analogue of the equality \eqref{Seki}:

\begin{theorem}[{\cite[Theorem 3.7]{Chang2017}}]\label{CM2017maindef}
For all $\textbf{s}\in \mathbb{N}_{>0}^r$, the equations
\begin{equation*}
\zeta_{\mathcal{A}_k}(\bold{s})=\frac{1}{\Gamma_{s_1}\cdots\Gamma_{s_r}}\sum_{\bold{j}\in\mathfrak{J}_\bold{s}}
\theta^{\bold{j}}\mathrm{Li}_{\mathcal{A}_k,\,\bold{s}}(u_{s_1,\,j_1},\dots,\,u_{s_r,\,j_r})
\end{equation*}
hold in $\mathcal{A}_k$.
\end{theorem}

Using elements $BC_n^{\bold{s},\,\bold{j}}$ of $k$, we can write down FMZVs as follows:

\begin{theorem}[{\cite[Theorem 32]{Harada2018}}]\label{FMZVandMPBCN}
\begin{enumerate}
\item
For $\textbf{s} \in \mathbb{N}_{>0}^r$, the congruence
\begin{equation}
\zeta_{\mathcal{A}_k}(\bold{s})_P
\equiv \frac{1}{\Gamma_{s_1}\cdots\Gamma_{s_r}}\sum_{\bold{j}\in\mathfrak{J}_\bold{s}}
\theta^{\bold{j}}  \sum_{d=r-1}^{\operatorname{deg}P-1}\frac{BC_{q^d-1}^{\bold{s},\,\bold{j}}}{L_d BC_{q^d-1}}  \label{eqFMZVandMPBCN}
\end{equation}
in the residue field $A/(P)$ holds for any $P \in \operatorname{Spm}A$ such that $P\not| \ \Gamma_{s_i}$ for $1\leq i\leq r$.
\item
Moreover,   if $r^\prime \in \mathbb{N}$ and $\bar{\bold{s}}=(1,\,\dots,1,\,s_1,\,\dots,\,s_r) \in \mathbb{N}^{r+r^\prime}$, the congruence
\begin{equation}
\zeta_{\mathcal{A}_k}(\bar{\bold{s}})_P
\equiv \frac{1}{\Gamma_{s_1}\cdots\Gamma_{s_r}}\sum_{\bold{j}\in\mathfrak{J}_\bold{s}}
\theta^{\bold{j}}  \sum_{\deg P > d_0>\cdots>d_{r^\prime}\geq r-1}\frac{BC_{q^{d_{r^\prime}}-1}^{\bold{s},\,\bold{j}}}{L_{d_0}\cdots L_{d_{r^\prime}} BC_{q^{d_{r^\prime}}-1}}  \label{eqFMZVandMPBCN2}
\end{equation}
in $A/(P)$ holds for any $P \in \operatorname{Spm}A$ such that $P\not| \ \Gamma_{s_i}$ for $1\leq i\leq r$.
\end{enumerate}
\end{theorem}
This is an analogue of Theorem \ref{thmImatomi}.

\subsection{Alternating multiple poly-Bernoulli-Carlitz numbers}\label{palternatingsection}

The purpose of this section is to extend the results in \cite{Harada2018} explained in \S \ref{Review on Harada's MPBCNs} to the alternating case.

\begin{definition}\label{DefMPBCN}
For $\bold{s}=(s_1,\dots,s_r) \in \mathbb{N}^r$, tuples $\boldsymbol{\gamma}=(\gamma_1,\,\dots,\,\gamma_r) \in (\overline{\mathbb{F}}_q^\times)^r$ of invertible elements of the algebraic closure of $\mathbb{F}_q$ and $\bold{j}=(j_1,\,\dots,\,j_r) \in\mathfrak{J}_{\bold{s}}$, the \textit{alternating multiple poly-Bernoulli-Carlitz numbers} (\textit{AMPBCNs} for short) $BC_n^{\bold{s},\,\boldsymbol{\gamma},\,\bold{j}}$ $ \in\bar{k}$ are defined by
\begin{equation*}
\sum_{n\geq0}BC_n^{\bold{s},\,\boldsymbol{\gamma},\,\bold{j}}\frac{z^n}{\Pi(n)}=\frac{\mathrm{Li}_{\bold{s}}(e_C(z)\gamma_1 u_{s_1,\,j_1},\,\gamma_2 u_{s_2,\,j_2},\dots,\,\gamma_r u_{s_r,\,j_r})}{e_C(z)}.
\end{equation*}
\end{definition}

This is an alternating extension of Definition \ref{DefMPBCNoriginal}.

\subsubsection{Connection with Stirling-Carlitz numbers}
We describe the above numbers as finite sums in terms of Stirling-Carlitz numbers, which could be regarded as an alternating extension of Theorem \ref{thmMPBCNandStirlingoriginal} and as an analogue of Theorem \ref{alBernoulliandStirling}.

\begin{theorem}\label{thmMPBCNandStirling}
If $\bold{s}$, $\boldsymbol{\gamma}$ and $\bold{j}$ are as in the Definition \ref{DefMPBCN}, the following equality holds:
\begin{equation*}
BC_n^{\bold{s},\,\boldsymbol{\gamma},\,\bold{j}}=\sum_{\log_q(n+1)\geq d_1>\cdots>d_r\geq0}\Gamma_{q^{d_1}}
\left\{
    \begin{array}{c}
     n\\
     q^{d_1}-1
    \end{array}
  \right\}_C \frac{(\gamma_1u_{s_1,j_1})^{q^{d_1}}\cdots(\gamma_r u_{s_r,j_r})^{q^{d_r}}}{L_{d_1}^{s_1}\cdots L_{d_r}^{s_r}}.
\end{equation*}
\end{theorem}

\begin{proof}
We have
\begin{align*}
&\frac{\mathrm{Li}_{\bold{s}}(e_C(z)\gamma_1 u_{s_1,\,j_1},\,\gamma_2 u_{s_2,\,j_2},\dots,\,\gamma_r u_{s_r,\,j_r})}{e_C(z)}\\
=&\sum_{d_1>\cdots>d_r\geq0}e_C(z)^{q^{d_1}-1}\frac{(\gamma_1u_{s_1,j_1})^{q^{d_1}}\cdots(\gamma_r u_{s_r,j_r})^{q^{d_r}}}{L_{d_1}^{s_1}\cdots L_{d_r}^{s_r}}\\
=&\sum_{d_1>\cdots>d_r\geq0} \left( \sum_{n\geq0}\Gamma_{q^{d_1}}
\left\{
    \begin{array}{c}
     n\\
     q^{d_1}-1
    \end{array}
  \right\}_C
\frac{z^n}{\Pi(n)}
\frac{(\gamma_1u_{s_1,j_1})^{q^{d_1}}\cdots(\gamma_r u_{s_r,j_r})^{q^{d_r}}}{L_{d_1}^{s_1}\cdots L_{d_r}^{s_r}}\right)\\
=& \sum_{n\geq0}\left( \sum_{d_1>\cdots>d_r\geq0}\Gamma_{q^{d_1}}
\left\{
    \begin{array}{c}
     n\\
     q^{d_1}-1
    \end{array}
  \right\}_C
\frac{(\gamma_1u_{s_1,j_1})^{q^{d_1}}\cdots(\gamma_r u_{s_r,j_r})^{q^{d_r}}}{L_{d_1}^{s_1}\cdots L_{d_r}^{s_r}}\right)\frac{z^n}{\Pi(n)}\\
=& \sum_{n\geq0}\left( \sum_{\log_q(n+1)\geq d_1>\cdots>d_r\geq0}\Gamma_{q^{d_1}}
\left\{
    \begin{array}{c}
     n\\
     q^{d_1}-1
    \end{array}
  \right\}_C
\frac{(\gamma_1u_{s_1,j_1})^{q^{d_1}}\cdots(\gamma_r u_{s_r,j_r})^{q^{d_r}}}{L_{d_1}^{s_1}\cdots L_{d_r}^{s_r}}\right)\frac{z^n}{\Pi(n)};
\end{align*}
the second equality follows from the definition of Stirling-Carlitz numbers and the forth holds by the equality \eqref{Stirling2}. Then the comparing coefficients of $z^n$ for each $n$ results in the desired equalities.
\end{proof} 
Using Theorem \ref{thmMPBCNandStirling} and the equality \eqref{Stirling2}, we obtain the following:
\begin{equation}
BC_{q^m-1}^{\bold{s},\,\boldsymbol{\gamma},\,\bold{j}}=\Gamma_{q^{m}}\sum_{m> d_2>\cdots>d_r\geq0}\frac{(\gamma_1u_{s_1,j_1})^{q^{m}}\cdots(\gamma_r u_{s_r,j_r})^{q^{d_r}}}{L_{m}^{s_1}\cdots L_{d_r}^{s_r}}. \label{MOBCNcorollary}
\end{equation}
where $m \in \mathbb{N}$, which is a generalization of \cite[Corollary 28]{Harada2018}.

\subsubsection{Connection with finite alternating multiple zeta values}

\begin{definition}\label{FAMZV}
For $\textbf{s}=(s_1,\dots,s_r) \in \mathbb{Z}^r$ and $\boldsymbol{\epsilon}=(\epsilon_1,\,\dots,\epsilon_r) \in {(A^\times)}^r$, the \textit{alternating finite multiple zeta value} (\textit{AFMZV} for short) $\zeta_{\mathcal{A}_k}(\bold{s}\,;\boldsymbol{\epsilon})$ $=$ $(\zeta_{\mathcal{A}_k}(\textbf{s}\,;\boldsymbol{\epsilon})_P)_{P\in \operatorname{Spm}A}$ $\in$ $ \mathcal{A}_k$ is defined by
\begin{equation*}
\zeta_{\mathcal{A}_k}(\textbf{s}\,;\boldsymbol{\epsilon})_P :=\sum_{\substack{\deg P>\deg{a_1}>\cdots>\deg a_r\geq0\\a_i\text{ monic}}}\frac{{\epsilon_1}^{\operatorname{deg}a_1}{\epsilon_2}^{\operatorname{deg}a_2}\cdots{\epsilon_r}^{\operatorname{deg}a_r}}{a_1^{s_1}\,\cdots a_r^{s_r}} \in A/(P).
\end{equation*}
\end{definition}

It can be seen as a characteristic $p$ analogue of AFMZV.
It immediately follows from \cite[Theorem 2.6]{Harada2020} that the product of any two AFMZVs are $\mathbb{F}_q$-linear combination of AFMZVs.

In order to obtain an alternating extension of Theorem \ref{CM2017maindef}, we extend the domain of FCMPLs from $\mathcal{A}_k$ to the ring $\mathcal{A}_{k^\prime}$ defined as follows: Let $q^\prime$ be a power of $q$. We define $A^\prime$, $k^\prime$ and $\mathcal{A}_{k^\prime}$ by the same ways as those of $A$, $k$ and $\mathcal{A}_k$ but substituting $q$ by $q^\prime$, and regard $A$, $k$ as subrings of $A^\prime$, $k^\prime$ by canonical ways, respectively. 
For each element $(a_P)$ of $\prod_P(A/(P))$ and each irreducible monic polynomial $Q_1$ in $A^\prime$ above $P_1 \in \operatorname{Spm}A$, define $b_{Q_1}$ to be the image of $a_{P_1}$ under the canonical embedding $A/(P_1)\hookrightarrow A^\prime/(Q_1)$ induced by the inclusion $A\rightarrow A^\prime$. Then the ring homomorphism from $\prod_P(A/(P))$ to $\prod_{Q \in \operatorname{Spm}A^\prime}(A^\prime/(Q))$ which maps $(a_P)$ to $(b_Q)$ induces an embedding of $\mathcal{A}_k$ into $\mathcal{A}_{k^\prime}$.

The FCMPLs can be extended to the multivariable functions on $\mathcal{A}_{k^\prime}$; for any tuples $\textbf{s}=(s_1,s_2,\dots,s_r) \in \mathbb{N}^r$ and $\bold{b}=((b_{Q,\,1}),\,\dots,\,(b_{Q,\,r})) \in \mathcal{A}_{k^\prime}^r$ with $(b_{Q,\,i}) \in A^\prime/Q$, we define the value $\mathrm{Li}_{\mathcal{A}_{k^\prime},\,\bold{s}}(\bold{b})=(\mathrm{Li}_{\mathcal{A}_k,\,\bold{s}}(\bold{b})_Q)_{Q\in\operatorname{Spm}A^\prime}$ by
\begin{equation*}
\mathrm{Li}_{\mathcal{A}_k,\,\bold{s}}(\bold{b})_Q:=\sum_{\deg P>i_1>\cdots>i_r\geq0} \frac{b_{Q,\,1}^{q^{i_1}}\,\cdots \,b_{Q,\,r}^{q^{i_r}}}{L_{i_1}^{s_1}\,\cdots \, L_{i_r}^{s_r}} \in A^\prime/(Q),
\end{equation*}
where the symbol $P$ in the right hand side stands for the monic irreducible polynomial in $A$ which is divided by $Q$ in $A^\prime$.

So far we put $q^\prime:=q^{q-1}$. We note that, for any $\epsilon \in A^\times=\mathbb{F}_q^\times$, the set ${A^{\prime}}^\times$ contains all $(q-1)$-th roots of $\epsilon$.
\begin{theorem}\label{FAMZVandMCPL}
Let $q^\prime:=q^{q-1}$ and let $\textbf{s}$ and $\boldsymbol{\epsilon}$ be as in the Definition \ref{FAMZV} and $\gamma_1,\,\dots,\gamma_r \in {A^{\prime}}^\times$ be $(q-1)$-th roots of $\epsilon_1,\,\dots,\epsilon_r$, respectively. Then, the equality
\begin{equation}
\zeta_{\mathcal{A}_k}(\bold{s}\,;\,\boldsymbol{\epsilon})
=\frac{1}{\gamma_1\Gamma_{s_1}\cdots\gamma_r\Gamma_{s_r}}\sum_{\bold{j}\in\mathfrak{J}_\bold{s}}
\theta^{\bold{j}}\mathrm{Li}_{\mathcal{A}_k,\,\bold{s}}(\gamma_1 u_{s_1,\,j_1},\dots,\,\gamma_r u_{s_r,\,j_r})\label{eqFAMZVandMCPL}
\end{equation}
in $\mathcal{A}_{k}$ holds.
\end{theorem}

Though the elements $\gamma_1,\,\dots,\gamma_r$ are not in $\mathcal{A}_k$ but in $\mathcal{A}_{k^\prime}$, we see that the right hand side of the equality \eqref{eqFAMZVandMCPL} is in $\mathcal{A}_k$ as the left hand side is.

\begin{proof}
It is enough to show that the congruences
\begin{align*}
\zeta_{\mathcal{A}_k}(\bold{s}\,;\,\boldsymbol{\epsilon})_P
\equiv\frac{1}{\gamma_1\Gamma_{s_1}\cdots\gamma_r\Gamma_{s_r}}\sum_{\bold{j}\in\mathfrak{J}_\bold{s}}
\theta^{\bold{j}}\mathrm{Li}_{\mathcal{A}_k,\,\bold{s}}(\gamma_1 u_{s_1,\,j_1},&\dots,\,\gamma_r u_{s_r,\,j_r})_P \mod P
\end{align*}
in $A^\prime/(P)\simeq\prod_{Q|P} A^\prime/(Q)$ hold for all but finite irreducible polynomial $P$ in $A$. Let $P$ be an element of $\operatorname{Spm}A$ such that $P \not | \ \Gamma_{s_i}$ for all $i$. We have the equalities and congruences:
\begin{align*}
&\zeta_{\mathcal{A}_k}(\bold{s}\,;\,\boldsymbol{\epsilon})_P\\
=&\sum_{\deg P >d_1>\cdots>d_r \geq0}\epsilon_1^{d_1}S_{d_1}(s_1)\cdots \epsilon_2^{d_2} S_{d_r}(s_r)\\
\equiv&\frac{1}{\Gamma_{s_1}\cdots\Gamma_{s_r}}\sum_{\deg P >d_1>\cdots>d_r \geq0}\frac{\epsilon_1^{d_1}H_{s_1 -1}^{(d_1)}(\theta)\cdots\epsilon_r^{d_r}H_{s_r -1}^{(d_r)}(\theta)}{L_{d_1}^{s_1}\cdots L_{d_r}^{s_r}} \mod P\\
=&\frac{1}{\Gamma_{s_1}\cdots\Gamma_{s_r}}\sum_{\deg P >d_1>\cdots>d_r \geq0}\sum_{\bold{j}\in\mathfrak{J}_\bold{s}}\theta^{\bold{j}}\frac{\epsilon_1^{d_1}u_{s_1,j_1}^{q^{d_1}}\cdots\epsilon_r^{d_r} u_{s_r,j_r}^{q^{d_r}}}{L_{d_1}^{s_1}\cdots L_{d_r}^{s_r}}\\
\equiv&\frac{1}{\gamma_1\Gamma_{s_1}\cdots \gamma_r\Gamma_{s_r}}\sum_{\bold{j}\in\mathfrak{J}_\bold{s}}\theta^{\bold{j}}\sum_{\deg P >d_1>\cdots>d_r \geq0}\frac{(\gamma_1u_{s_1,j_1})^{q^{d_1}}\cdots(\gamma_r u_{s_r,j_r})^{q^{d_r}}}{L_{d_1}^{s_1}\cdots L_{d_r}^{s_r}} \mod P\\
=&\frac{1}{\gamma_1\Gamma_{s_1}\cdots \gamma_r\Gamma_{s_r}}\sum_{\bold{j}\in\mathfrak{J}_\bold{s}}\theta^{\bold{j}}\mathrm{Li}_{\mathcal{A}_k,\,\bold{s}}(\gamma_1 u_{s_1,\,j_1},\,\dots,\,\gamma_r u_{s_r,\,j_r})_P,
\end{align*}
where the second congruence is by equations $\gamma_{i}^{q^d}=\epsilon^i \gamma_i$ which holds for $1\leq i\leq r$ and $d\geq0$.
\end{proof}

The following lemma is an alternating extension of \cite[Lemma 31]{Harada2018}.
\begin{lemma}\label{MPBCNrecursion}If we take $\bold{s}$, $\boldsymbol{\gamma}$ and $\bold{j}$ as in the Definition \ref{DefMPBCN}, then the recursive formula 
\begin{equation*}
BC_{q^m-1}^{\bold{s},\,\boldsymbol{\gamma},\,\bold{j}}=BC_{q^m-1}^{s_1,\gamma_1,\,\,j_1}\sum_{d=r-2}^{m-1}\frac{1}{\Gamma_{q^d}}BC_{q^d-1}^{\bold{s}^*,\,\boldsymbol{\gamma}^*,\,\bold{j}^*}
\end{equation*}
holds for $m \in \mathbb{N}_{>0}$ where $\bold{s}^*:=(s_2,\,\dots,\,s_r)$, $\boldsymbol{\gamma}^*:=(\gamma_2,\,\dots\,\gamma_r)$ and  $\bold{j}^*:=(j_2,\,\dots,\,j_r)$.
\end{lemma}

\begin{proof}
The assertion is obtained as follows:

\begin{align*}
BC_{q^m-1}^{\bold{s},\,\boldsymbol{\gamma},\,\bold{j}}&=\sum_{m> d_2>\cdots>d_r\geq0}\Pi(q^{m}-1)\frac{(\gamma_1u_{s_1,j_1})^{q^{m}}\cdots(\gamma_r u_{s_r,j_r})^{q^{d_r}}}{L_{m}^{s_1}\cdots L_{d_r}^{s_r}}\\
&=\Pi(q^m-1)\frac{(\gamma_1 u_{s_1,\,j_1})^{q^m}}{L_m^{s_1}} \sum_{m> d_2>\cdots>d_r\geq0}\frac{(\gamma_2 u_{s_2,\,j_2})^{q^{d_2}}\cdots(\gamma_r u_{s_r,j_r})^{q^{d_r}}}{L_{d_2}^{s_2}\cdots L_{d_r}^{s_r}}\\
&=BC_{q^m-1}^{s_1,\gamma_1,\,\,j_1}\sum_{m> d_2>\cdots>d_r\geq0}\frac{(\gamma_2 u_{s_2,\,j_2})^{q^{d_2}}\cdots(\gamma_r u_{s_r,j_r})^{q^{d_r}}}{L_{d_2}^{s_2}\cdots L_{d_r}^{s_r}}\\
&=BC_{q^m-1}^{s_1,\gamma_1,\,\,j_1} \sum_{d_2=r-2}^{m-1}\frac{1}{\Gamma_{q^{d_2}}}\sum_{d_2>d_3>\cdots>d_r\geq0}\Gamma_{q^{d_2}}\frac{(\gamma_2 u_{s_2,\,j_2})^{q^{d_2}}\cdots(\gamma_r u_{s_r,j_r})^{q^{d_r}}}{L_{d_2}^{s_2}L_{d_3}^{s_3}\cdots L_{d_r}^{s_r}}\\
&=BC_{q^m-1}^{s_1,\gamma_1,\,\,j_1}\sum_{d=r-2}^{m-1}\frac{1}{\Gamma_{q^d}}BC_{q^d-1}^{\bold{s}^*,\,\boldsymbol{\gamma}^*,\,\bold{j}^*};
\end{align*}
where the third and the fifth equalities are due to the equality \eqref{MOBCNcorollary}.
\end{proof}

The following could be seen as  an alternating extension of Theorem \ref{FMZVandMPBCN} and also seen as an analogue of Theorem \ref{0result}.
\begin{theorem}\label{FAMZVandMPBCN}We put $q^\prime :=q^{q-1}$. The followings hold.
\begin{enumerate}
\item
If we take $\textbf{s}$ and $\boldsymbol{\epsilon}$ be as in the Definition \ref{FAMZV}
and any $(q-1)$-th roots $\gamma_1,\,\dots,\gamma_r \in \mathbb{F}_{q^\prime}$ of $\epsilon_1,\dots,\epsilon_r$, respectively, then the congruences
\begin{equation}
\zeta_{\mathcal{A}_k}(\bold{s}\,;\,\boldsymbol{\epsilon})_P
\equiv \frac{1}{\gamma_1\Gamma_{s_1}\cdots\gamma_r\Gamma_{s_r}}\sum_{\bold{j}\in\mathfrak{J}_\bold{s}}
\theta^{\bold{j}}  \sum_{d=r-1}^{\operatorname{deg}P-1}\frac{BC_{q^d-1}^{\bold{s},\,\boldsymbol{\gamma},\,\bold{j}}}{L_d BC_{q^d-1}}  \label{eqFAMZVandMPBCN}
\end{equation}
in the residue ring $A^\prime/(P)$ hold for all $P \in \operatorname{Spm}A$ such that $P\not| \  \Gamma_{s_i}$ for $1\leq i\leq r$.
\item
For ${r^\prime}\in\mathbb{N}$, we put $\bar{\bold{s}}=(1,\,\dots,1,\,s_1,\,\dots,\,s_r) \in \mathbb{N}^{r+{r^\prime}}$ and $\bar{\boldsymbol{\epsilon}}=$ $(1,\,$$\dots,\,1,$$\,\epsilon_1,$$\,\dots$$ ,$ $\epsilon_r) \in (\mathbb{F}_q^\times)^{r+{r^\prime}}$. Then the congruences 
\begin{equation*}
\zeta_{\mathcal{A}_k}(\bar{\bold{s}}\,;\,\bar{\boldsymbol{\epsilon}})_P
\equiv \frac{1}{\gamma_1\Gamma_{s_1}\cdots\gamma_r\Gamma_{s_r}}\sum_{\bold{j}\in\mathfrak{J}_\bold{s}}
\theta^{\bold{j}}  \sum_{\deg P > d_0>\cdots>d_{r^\prime}\geq r-1}\frac{BC_{q^{d_{r^\prime}}-1}^{\bold{s},\,\boldsymbol{\gamma},\,\bold{j}}}{L_{d_0}\cdots L_{d_{r^\prime}} BC_{q^{d_{r^\prime}}-1}}  \label{eqFAMZVandMPBCN}
\end{equation*}
in $A^\prime/(P)$ hold for all $P \in \operatorname{Spm}A$ such that $P\not| \ \Gamma_{s_i}$ for $1\leq i\leq r$.
\end{enumerate}
\end{theorem}

\begin{proof}
 We have
\begin{align*}
&\gamma_1\Gamma_{s_1}\cdots\gamma_r\Gamma_{s_r}\zeta_{\mathcal{A}_k}(\bold{s}\,;\,\boldsymbol{\epsilon})_P\\
=&\sum_{\bold{j}\in\mathfrak{J}_\bold{s}} \theta^{\bold{j}}\mathrm{Li}_{\mathcal{A}_k,\,\bold{s}}(\gamma_1 u_{s_1,\,j_1},\dots,\,\gamma_r u_{s_r,\,j_r})_P\\
=&\sum_{\bold{j}\in\mathfrak{J}_\bold{s}} \theta^{\bold{j}}\sum_{\operatorname{deg}P> d_1>\cdots>d_r\geq0}\frac{(\gamma_1u_{s_1,j_1})^{q^{d_1}}\cdots(\gamma_r u_{s_r,j_r})^{q^{d_r}}}{L_{d_1}^{s_1}\cdots L_{d_r}^{s_r}}\\
=&\sum_{\bold{j}\in\mathfrak{J}_\bold{s}} \theta^{\bold{j}} \sum_{d=r-1}^{\deg P -1}\frac{1}{\Gamma_{q^d}}\sum_{\operatorname{deg}P> d>\cdots>d_r\geq0}\Gamma_{q^d}\frac{(\gamma_1u_{s_1,j_1})^{q^{d}}\cdots(\gamma_r u_{s_r,j_r})^{q^{d_r}}}{L_{d}^{s_1}\cdots L_{d_r}^{s_r}}\\
=&\sum_{\bold{j}\in\mathfrak{J}_\bold{s}} \theta^{\bold{j}} \sum_{d=r-1}^{\deg P -1}\frac{1}{\Gamma_{q^d}}BC_{q^d-1}^{\bold{s},\,\boldsymbol{\gamma},\,\bold{j}}
=\sum_{\bold{j}\in\mathfrak{J}_\bold{s}}
\theta^{\bold{j}}  \sum_{d=r-1}^{\operatorname{deg}P-1}\frac{BC_{q^d-1}^{\bold{s},\,\boldsymbol{\gamma},\,\bold{j}}}{L_d BC_{q^d-1}},
\end{align*}
for such a $P$, hence we obtain the first assertion. The last equality is from the following equation
\begin{equation*}
\frac{BC_{q^d-1}}{\Gamma_{q^d}}=\frac{1}{L_d}
\end{equation*}
which holds for each $d\geq0$; this is from the equality \eqref{MOBCNcorollary}.

To show the second assertion, take $P$ such that $P \not| \ \Gamma_{s_i}$ for all $i$. If $\bar{\boldsymbol{\gamma}}$ stands for the tuple $(1,\,\dots,1,\,\gamma_1,\dots,\gamma_r) \in  ({\overline{\mathbb{F}}_q^\prime}^\times)^{r+{r^\prime}}$, the assertion (1) of Theorem \ref{FAMZVandMPBCN} yields the equality
\begin{align*}
\zeta_{\mathcal{A}_k}(\bar{\bold{s}}\,;\,\bar{\boldsymbol{\epsilon}})_P
&=\frac{1}{\Gamma_1^{r^\prime} \gamma_1\Gamma_{s_1}\cdots\gamma_r\Gamma_{s_r}}\sum_{\bold{j}\in\mathfrak{J}_{\bar{\bold{s}}}}
\theta^{\bold{j}}  \sum_{d_0=r-1}^{\operatorname{deg}P-1}\frac{BC_{q^{d_0}-1}^{\bar{\bold{s}},\,\bar{\boldsymbol{\gamma}},\,\bar{\bold{j}}}}{L_{d_0} BC_{q^{d_0}-1}}\\
&=\frac{1}{\gamma_1\Gamma_{s_1}\cdots\gamma_r\Gamma_{s_r}}\sum_{\bold{j}\in\mathfrak{J}_{\bar{\bold{s}}}}
\theta^{\bold{j}}  \sum_{d_0=r-1}^{\operatorname{deg}P-1}\frac{BC_{q^{d_0}-1}^{\bar{\bold{s}},\,\bar{\boldsymbol{\gamma}},\,\bar{\bold{j}}}}{\Pi(q^{d_0}-1)}\\
&=\frac{1}{\gamma_1\Gamma_{s_1}\cdots\gamma_r\Gamma_{s_r}}\sum_{\bold{j}\in\mathfrak{J}_{{\bold{s}}}}
\theta^{\bold{j}}  \sum_{d_0=r-1}^{\operatorname{deg}P-1}\frac{BC_{q^{d_0}-1}^{{\bold{s}},\,{\boldsymbol{\gamma}},\,{\bold{j}}}}{\Pi(q^{d_0}-1)}.
\end{align*}
since $H_0=1$.
By applying the Lemma \ref{MPBCNrecursion} repeatedly, we can calculate as follows for $d_0 \geq r^\prime +r-1$:
\begin{align*}
\frac{BC_{q^{d_0}-1}^{\bar{\bold{s}},\,\bar{\boldsymbol{\gamma}},\,\bar{\bold{j}}}}{\Pi(q^{d_0}-1)}
&=\frac{BC^{1,1,0}_{q^{d_0}-1}}{\Gamma_{q^{d_0}}} \sum_{d_1=r+r^\prime-2}^{d_0-1}\frac{BC^{\bar{\bold{s}}^*,\,\bar{\boldsymbol{\gamma}}^*,\,\bar{\bold{j}}^*}_{q^{d_1}-1}}{\Gamma_{q^{d_1}}}\\
&=\frac{BC^{1,1,0}_{q^{d_0}-1}}{\Gamma_{q^{d_0}}} \sum_{d_1=r+r^\prime-2}^{d_0-1}\frac{BC^{1,1,0}_{q^{d_1}-1}}{\Gamma_{q^{d_1}}} \sum_{d_2=r+r^\prime-2}^{d_1-1}\frac{BC^{\bar{\bold{s}}^{**},\,\bar{\boldsymbol{\gamma}}^{**},\,\bar{\bold{j}}^{**}}_{q^{d_2}-1}}{\Gamma_{q^{d_2}}}\\
&\vdots \\
&=\sum_{d_0>\cdots>d_{r^\prime}>r-1}\left(\prod_{i=0}^{r^\prime} \frac{BC_{q^{d_i}-1}^{1,1,0}}{\Gamma_{q^{d_i}}} \right)\frac{BC_{q^{d_{r^\prime}}-1}^{\bold{s},\,\boldsymbol{\gamma},\,\bold{j}}}{ BC_{q^{d_{r^\prime}}-1}}\\
&=\sum_{d_0>\cdots>d_{r^\prime}>r-1}\left(\prod_{i=0}^{r^\prime} \frac{1}{L_{d_i}} \right)\frac{BC_{q^{d_{r^\prime}}-1}^{\bold{s},\,\boldsymbol{\gamma},\,\bold{j}}}  { BC_{q^{d_{r^\prime}}-1}}.\\
\end{align*}
Hence we obtain the desired equality.
\end{proof}

\subsubsection*{Acknowledgment}

The author is deeply grateful to Professor H. Furusho; without his profound instruction and continuous encouragements the present paper would never be accomplished. He is also grateful to R. Harada who guided him to the research on positive characteristic arithmetic.

\appendix
\section{Finite multiple zeta values with non-all-positive indices} \label{Finite multiple zeta values with non-all-positive indices}

In the characteristic $0$ case, it is known that any FMZV with integer index is expressed as $\mathbb{Q}$-linear combination of FMZV's with all-positive indices (cf \cite{Kaneko2019}). Here, we show that the same is true in the case of characteristic $p$ (Theorem \ref{Appthm}).

We recall that the sum 
$
\sum_{a} \frac{1}{a^s}
$
(where $a$ runs through all monic polynomials of degree $d$ in $A$) is denoted by $S_d(s)$ (cf. \S\ref{Review on Harada's MPBCNs}). The following is a special case of \cite[Proposition 4.1]{Goss1979}:

\begin{prop}\label{Appprop}
For $s\in\mathbb{N}_{\geq0}$, there is $N(s) \in \mathbb{N}$ such that $S_d(-s)=0$ for any $d\geq N(s)$.
\end{prop}
\begin{proof}
If $s=0$, it is enough to put $N(s)=1$ since a number of elements of the set of all monic polynomials of degree $d$ is $q^d$ for $d\in\mathbb{N}$.

For general $s$, it is enough to out $N(s):=\max\{ N(t)+1 \mid t<s\}$.
Indeed, for $d\geq N(s)$ we have
\begin{align*}
S_d(-s)&=\sum_{\substack{\deg a =d-1\\a:\text{monic}\\b\in\mathbb{F}_q}}(\theta a+b)^s
=\sum_{\substack{\deg a =d-1\\a:\text{monic}\\b\in\mathbb{F}_q}} \sum_{t=0}^s (\theta a)^t b^{s-t}
\left(
    \begin{array}{c}
     s\\
     t
    \end{array}
  \right)\\ 
&=\theta^s\sum_{\substack{\deg a =d-1\\a:\text{monic}}} a^s \sum_{b\in\mathbb{F}_q}b^0=0.
\end{align*}
\end{proof}

We need the following lemma:
\begin{lemma}\label{Applemma}
For any tuple $(s_1,\,\dots,\,s_r)\in \mathbb{Z}^r,\,0\leq M\leq r$ and $N\in\mathbb{N}$, the element
\begin{equation*}
\left(\sum_{\deg P >d_1>\cdots >d_M\geq N>d_{M+1}>\cdots>d_r\geq0}S_{d_1}(s_1)\cdots S_{d_r}(s_r)\right)_P
\end{equation*}
of $\mathcal{A}_k$ is a $k$-linear combination of FMZVs with depth equal to or less than $r$.
\end{lemma}

\begin{proof}
This is proven by the induction on depth $r$. If $M<r$, we have 

\begin{align*}
&\sum_{\deg P >d_1>\cdots >d_M\geq N>d_{M+1}>\cdots>d_r\geq0}S_{d_1}(s_1)\cdots S_{d_r}(s_r)\\
=&\sum_{N >d_{M+1}>\cdots>d_r\geq0}S_{d_{M+1}}(s_{M+1})\cdots S_{d_r}(s_r)\times \sum_{\deg P >d_1>\cdots>d_M\geq N}S_{d_1}(s_1)\cdots S_{d_M}(s_M),
\end{align*}
hence the induction hypothesis implies the desired result. In the case $M=r$, the equation
\begin{align*}
&\sum_{\deg P >d_1>\cdots >d_r\geq N}S_{d_1}(s_1)\cdots S_{d_r}(s_r)\\
=&\zeta_\mathcal{A}(s_1,\,\dots,\,s_r)_P - \sum_{M^\prime=0}^{r-1}\left(\sum_{\deg P >d_1>\cdots >d_{M^\prime} \geq N>d_{M^\prime+1}>\cdots>d_r\geq0}S_{d_1}(s_1)\cdots S_{d_r}(s_r)\right)
\end{align*}
holds. Therefore we have the result.

\end{proof}

\begin{theorem} \label{Appthm}
Any FMZV with integer index is expressed as $k$-linear combinations of FMZVs with all positive indices.
\end{theorem}

\begin{proof}
We use the induction on depth. We consider a FMZV $\zeta_\mathcal{A}(s_1,\,\dots,\,s_r)$. If $s_1\leq0$, then Proposition \ref{Appprop} implies that $\zeta_\mathcal{A}(s_1,\,\dots,\,s_r) \in k$. Assume that $s_{M+1}\leq0$ for some $M$ with $1\leq M\leq r-1$. Then we can take $N \in \mathbb{N}$ such that $S_d(s_{M+1})=0$ for any $d\geq M$. Then we have
\begin{equation*}
\zeta_\mathcal{A}(s_1,\,\dots,\,s_r)=\left(\sum_{\deg P >d_1>\cdots >d_M\geq N>d_{M+1}>\cdots>d_r\geq0}S_{d_1}(s_1)\cdots S_{d_r}(s_r)\right)_P.
\end{equation*}
Hence, the desired result follows from the induction hypothesis and Lemma \ref{Applemma}.
\end{proof}

\begin{remarks}
\begin{enumerate}
\item
In the same way, we can show that any AFMZV with integer index can be expressed as $k$-linear combinations of AFMZVs with all positive indices.
\item
If $s_i\leq0$ for all $i$, a FMZV $\zeta_\mathcal{A}(s_1,\,\dots,\,s_r)$ is in $A$.
\end{enumerate}
\end{remarks}



\end{document}